\documentclass[a4paper,twoside]{article}
\usepackage[utf8]{inputenc}
\usepackage{enumitem}
\usepackage[T1]{fontenc}
\usepackage{amsmath}
\usepackage{amsfonts}
\usepackage{amssymb}
\usepackage[left=4cm,right=4cm,top=4cm,bottom=4cm]{geometry}
\usepackage{amsthm}
\usepackage{graphicx}
\usepackage{caption}
\usepackage{subcaption}
\usepackage{mathtools}
\usepackage{tikz}
\usepackage{import}
\usepackage{fancyhdr}
\usepackage[hidelinks]{hyperref}
\usepackage{float}
\usepackage[procnames]{listings}

\usetikzlibrary{matrix}

\hypersetup{
  colorlinks   = true, 
  urlcolor     = black,
  linkcolor    = red, 
  citecolor   = blue 
}

\graphicspath{{drawings/}}

\newcommand{\Q}{\ensuremath{\mathbb{Q}}}
\newcommand{\Z}{\ensuremath{\mathbb{Z}}}
\newcommand{\N}{\ensuremath{\mathbb{N}}}

\newcommand{\sinf}{<}
\newcommand{\ssup}{>}

\newcommand{\Val}{\text{Val}}

\DeclareRobustCommand{\qbinom}{\genfrac[]{0pt}{}}

\theoremstyle{plain}
\newtheorem{thm}{Theorem}
\newtheorem{prop}[thm]{Proposition}
\newtheorem{lemma}[thm]{Lemma}
\newtheorem{cor}[thm]{Corollary}
\newtheorem{res}{Result}

\theoremstyle{definition}
\newtheorem{defn}[thm]{Definition}
\newtheorem*{note}{Nota bene}
				
\theoremstyle{remark}
\newtheorem{rem}[thm]{Remark}
\newtheorem{ex}[thm]{Example}

\newcommand{\address}{{
  \bigskip
  \footnotesize

  \textsc{Institut de Mathématiques de Toulouse, UMR5219,}\par\nopagebreak
  \textsc{UPS, F-31062 Toulouse Cedex 9, France}\par\nopagebreak
  \textit{E-mail address}: \texttt{sonny.willetts@math.univ-toulouse.fr}

}}

\fancypagestyle{mypagestyle}{
  \fancyhf{}
  \fancyhead[OC]{\textsc{s.willetts}}
  \fancyhead[EC]{\textsc{Vassiliev invariants for knots from the ADO polynomials}}
  \fancyfoot[C]{\thepage}
  
}
\begin{document}
\title{\textbf{\textsc{Vassiliev invariants for knots from the ADO polynomials}}}
\author{\textsc{Sonny Willetts}}
\date{}
\maketitle


\begin{abstract}
In this paper we prove that the $r$-th ADO polynomial of a knot, for $r$ a power of prime number, can be expanded as Vassiliev invariants with values in $\Z$. Nevertheless this expansion is not unique and not easily computable. We can obtain a unique computable expansion, but we only get $r$ adic topological Vassiliev invariants as coefficients.  To do so, we exploit the fact that the colored Jones polynomials can be decomposed as Vassiliev invariants and we tranpose it to ADO using the unified knot invariant recovering both ADO and colored Jones defined in \cite{Willetts2020AUO}. Finally we prove some asymptotic behavior of the ADO polynomials modulo $r$ as $r$ goes to infinity.
\end{abstract}

\section*{Introduction}
Quantum knots invariants at formal $q$ can be decomposed $h$-adically as Vassiliev invariants by setting $q=e^h$. This is the case for the colored Jones polynomials for knots. Nevertheless, for the invariants coming from quantum algebra at roots of unity, it is unclear if they can be decomposed into Vassiliev invariants. As invariants coming from a version of quantum $\mathfrak{sl}_2$ at roots of unity, the ADO polynomials for knots don't have a $h$-adic decomposition from which coefficients would be Vassiliev invariants.

The strategy used in this paper to make Vassiliev invariants appear from ADO polynomials is to use the unified knot invariant defined in \cite{Willetts2020AUO}. This two variables invariant specializes to the colored Jones in some variable and into the ADO polynomials divided by the Alexader polynomial in the other variable. 

Since the colored Jones polynomials can be decomposed $h$-adically into Vassiliev invariants, we use this fact to obtain the following result for the unified invariant $F_{\infty}(q,q^{\alpha},\mathcal{K})$:

\begin{res}
Let $\mathcal{K}$ be a $0$ framed knot in $S^3$.\\
We can write in a unique way the unified invariant \[F_{\infty}(q,q^{\alpha},\mathcal{K})= \sum_{m,n \geq 0} b_{n,m}(\mathcal{K}) (q^2 -1 )^n (q^{2\alpha}-1)^m\]
And we have that, for any $m,n \in \N$, the map
\begin{align*}
 \nu_{n,m} : \mathsf{K} &\to \Z\\
 \mathcal{K} &\mapsto b_{n,m}(\mathcal{K})
 \end{align*}
 
 is a Vassiliev invariant of degree\textsuperscript{\hyperlink{NB_degree}{1}} $n+m$.\\
 (See Proposition \ref{Vassiliev_coeff_unified} for a more detailed version.)
\end{res}

The main issue to transpose this result to ADO is that we pass through a completion ring, and we must be careful as to how we evaluate at roots of unity. The evaluation at roots of unity will lie in a completion ring that will not be trivial if $r$ is a power of a prime number.

As a result we can study an expansion of the $r$-th ADO polynomial whose coefficients will be Vassiliev invariants:

\begin{res}\label{result_ADO_vassiliev}
Let $\mathcal{K}$ be a $0$ framed knot in $S^3$ and $r=p^l$ be a power of a prime number.\\
There exists $c_{n,m}(r,\mathcal{K}) \in \Z$ such that we can write the ADO polynomial as \[ADO_r(q^{\alpha},\mathcal{K})= \sum_{m,n \geq 0} c_{n,m}(r,\mathcal{K}) (\zeta_r -1 )^n (q^{2\alpha}-1)^m\]
Such that, for any $m,n \in \N$, the map
\begin{align*}
 \nu_{n,m} : \mathsf{K} &\to \Z\\
 \mathcal{K} &\mapsto c_{n,m}(r,\mathcal{K})
 \end{align*}
 
 is a Vassiliev invariant of degree\textsuperscript{\hyperlink{NB_degree}{1}} $n+m$.\\
 (See Proposition \ref{thm_ado_Vassiliev} for a more detailed version.)
\end{res}

\begin{rem}
The sum in Result \ref{result_ADO_vassiliev}, the sum is finite for each $\mathcal{K}$ but there is a priori no global bound.
\end{rem}

But this expansion is not unique and cannot be determined easily having only the ADO polynomial in hand. In order to have a unique computable expansion, we will no longer have Vassiliev invariant coefficients but $r$ adic topological Vassiliev invariant coefficients.

\begin{res}
Let $\mathcal{K}$ be a $0$ framed knot in $S^3$ and $r=p^l$ be a power of a prime number.

We can write in a unique way \[ADO_r(q^{\alpha}, \mathcal{K}) = \sum_{m \geq 0} \sum_{n=0}^{\varphi(r)-1} d_{n,m}(r,\mathcal{K}) (\zeta_r-1)^n (q^{2\alpha}-1)^m  \] in $\Z[\zeta_r, q^{\alpha}]$,
and we have  
\begin{enumerate}
\item $d_{n,m}(r,\mathcal{K}) \mod r^j$ is a Vassiliev invariant of degree\textsuperscript{\hyperlink{NB_degree}{1}} $(j+1)\varphi(r)+m$.
\item In consequence, $d_{n,m}(r,\mathcal{K}) \in \Z$ and $d_{n,m}(r,\mathcal{K})$ is a $r$-adic topological Vassiliev invariant.
\item Thus, $ADO_r(q^{\alpha}, \mathcal{K})\in \Z[\zeta_r, q^{\alpha}]$ is a topological Vassiliev invariant for the filtration $(r^j (q^{\alpha}-1)^m)_{j,m \in \N}$ topology.
\end{enumerate}

(See Theorem \ref{thm_ado_Vassiliev_radic} for a more detailed version)
\end{res}

Then, we can study these coefficients modulo $r$ in order to get some asymptotic result for the ADO polynomials.

\begin{res}
Let $r=p^l$ be a power of a prime number and $\mathcal{K}$ be a $0$ framed knot in $S^3$.\\
We can write in a unique way \[ADO_r(q^{\alpha}, \mathcal{K})= \sum_{m \geq 0} \sum_{n=0}^{\varphi(r)-1} d_{n,m}(r,\mathcal{K}) (\zeta_{r} -1 )^n (q^{2\alpha}-1)^m\]
And we have that, for any $m \in \N$ and any $1 \leq n \leq \varphi(r)-1$, the map
\begin{align*}
 \nu_{n,m} : \mathsf{K} &\to \Z/r\Z\\
 \mathcal{K} &\mapsto \overline{d_{n,m}(r,\mathcal{K})}
 \end{align*}
 
 is a Vassiliev invariant of degree\textsuperscript{\hyperlink{NB_degree}{1}} $n+m$.\\
 Moreover, \[ \forall n \in \N \text{, } \forall m \sinf r\text{, }  \overline{d_{n,m}(r,\mathcal{K})}= \overline{b_{n,m}(\mathcal{K})} \text{ in  }\Z/rZ \]
 And hence, \[\Val_{q^{\alpha}-1}(\overline{F_{\infty}(\zeta_r,q^{\alpha},\mathcal{K})- ADO_r(q^{\alpha}, \mathcal{K})}) \underset{r \to +\infty}{\longrightarrow} +\infty \] where \[\Val_{q^{\alpha}-1}(\sum_m a_m (q^{\alpha}-1)^m)=\min(m \in \N | \  a_m \neq 0)\] is the valuation map.
 
 (See Proposition \ref{thm_ADO_Vassiliev_r} and \ref{thm_asymptotic_ado} for more detailed versions.)
\end{res}

\begin{note}
Please bear in mind that, in this article, the results are only available for $0$ framed knots in $S^3$.
\end{note}


\section{The unified invariant}
In this section, we will briefly present the construction of the unified invariant developped in \cite{Willetts2020AUO} where the proofs and more details are given.

\paragraph{Completion ring:}~\\
Here we describe the ring in which the unified knot invariant lies.

Let $R=\Z[q^{\pm 1}, A^{\pm 1}]$, we construct a completion of that ring. For the sake of simplicity, we will denote $q^{\alpha} := A$ and use previous notation for quantum numbers. Keep in mind that, here, $\alpha$ is just a notation, not a complex number.

\noindent We denote $ \{ \alpha \}_q = q^{\alpha}-q^{-\alpha}$, $\{ \alpha +k \}_q = q^{\alpha }q^k - q^{-\alpha } q^{-k}$, $\{ \alpha;n \}_q= \prod_{i=0}^{n-1} \{ \alpha -i\}_q$, $\{n\}_q!=\{n;n\}_q $, $\qbinom{\alpha+i}{n}=\frac{\{\alpha+i;n\}_q}{\{n\}!}$.

\begin{defn}
Let $I_n$ be the ideal of $R$ generated by the following set $\{ \ \{ \alpha+l; n \}_q , \ l \in \Z \}$.
\end{defn}

We then have a projective system : \[ \hat{I} : I_1 \supset I_2 \supset \dots \supset I_n \supset \dots \]
From which we can define the completion of $R$, taking the projective limit:

\begin{defn}
Let $\hat{R} = \underset{\underset{n}{\leftarrow}}{\lim} \dfrac{R}{I_n} = \{ (a_n)_{n \in \N^*} \in \prod_{i=1}^{\infty} \frac{R}{I_n} \ | \ \rho_n(a_{n+1})=a_n \}$  where $\rho_n: \frac{R}{I_{n+1}} \to \frac{R}{I_{n}}$ is the projection map.
\end{defn}

We denote $\widehat{\Z[q,q^{\alpha}]}=\hat{R}$.

\paragraph{Completion algebra and universal invariant:}
\begin{defn}
We set $U_h := U_h(\mathfrak{sl}_2)$ the $\Q[[h]]$ algebra topologically generated by $H,E,F$ and relations \[ [H,E]=2E, \ [H,F]=-2F, \ [E,F]=\frac{K-K^{-1}}{q-q^{-1}} \]
where $q=e^h$ and $K=q^H= e^{hH}$.
\end{defn}

It is endowed with an Hopf algebra structure:
$$
\begin{array}{lll}
\Delta(E)=1 \otimes E + E \otimes K & \epsilon(E)=0 &S(E)=-EK^{-1} \\
\Delta(F)=K^{-1} \otimes F + F \otimes 1 & \epsilon(F)=0 &S(E)=-KF \\
\Delta(H)=1 \otimes H + H \otimes 1 & \epsilon(H)=0 &S(H)=-H \\
\end{array}
$$

And an $R$-matrix:

\[R= q^{\frac{H \otimes H}{2}}\underset{i=0}{\overset{\infty}{\sum}} \frac{\{1\}^{2n} q^{\frac{n(n-1)}{2}}}{\{n\}_q!} E^n \otimes F^{n}\]

Altogether with a ribbon element: $K^{-1} u$ where $u=\sum S(\beta) \alpha$ if $R= \sum \alpha \otimes \beta$.

\bigskip
\noindent Hence, if $\mathcal{K}$ is a knot, we set $Q^{U_h}(\mathcal{K}) \in U_h$ the universal invariant associated to $\mathcal{K}$ in $U_h$. The definition of this element is given in Ohtsuki's book \cite{ohtsuki2002quantum} subsection 4.2. It is a knot invariant.
\bigskip

\begin{defn}
Let $\mathcal{U}:=U_q^{D}(\mathfrak{sl}_2)$ be the $\Z[q^{\pm 1}]$ subalgebra of $U_h$ generated by $E,\  F^{(n)},\ K$ where $F^{(n)}= \frac{\{1\}^{2n} F^n }{\{n\}_q!}$.
\end{defn}

It inherits the Hopf algebra structure from $U_h$.

\medskip
\noindent We denote $\{H+m \}_q=Kq^m-K^{-m}q^{-m}$, $\{H+m;n \}_q= \prod_{i=0}^{n-1} \{H+m-i\}_q$.

\begin{defn}
Let $L_n$ be the $\Z[q^{\pm 1} ] $ ideal generated by $\{ n\}!$.\\
Let $J_n$ be the $\mathcal{U}$ two sided ideal generated by the following elements:
\[ F^{(i+k)} \{H+m;n-i \}_q \] where $m \in \Z$ and $i \in \{0, \dots, n \}$ and $k \in \N$.
\end{defn}

We define the completion $\mathcal{\hat{U}}:= \underset{\underset{n}{\leftarrow}}{\lim} \dfrac{\mathcal{U}}{J_n}$ which is a $\widehat{\Z[q]}:= \underset{\underset{n}{\leftarrow}}{\lim} \dfrac{\Z[q]}{L_n}$ Hopf algebra (see \cite{Willetts2020AUO}~\cite{habiro2007integral}). 

There is a map $i: \mathcal{\hat{U}} \to U_h$.
Since we do not know if this map is injective, we consider $\tilde{\mathcal{U}}:= i(\mathcal{\hat{U}})$ the image in $U_h$. It is also an Hopf algebra.

We then have the following proposition:
\begin{prop}
If $\mathcal{K}$ is a $0$ framed knot, the universal invariant in $U_h$ lies in $ \mathcal{\tilde{U}}$, and we define $Q^{\mathcal{\tilde{U}}} (\mathcal{K})$ to be such that:
\[ Q^{U_h}(\mathcal{K})= Q^{\mathcal{\tilde{U}}} (\mathcal{K}) \in  \mathcal{\tilde{U}} \]
\end{prop}

See subsection 4.1 of \cite{Willetts2020AUO} for more details.

\paragraph{The topological Verma type module:}~\\ 
These completions allows us to construct a topological Verma type module on $\mathcal{\tilde{U}}$  whose coefficient ring will be $\widehat{\Z[q,q^{\alpha}]}$.

Let $V^{\alpha}$ be a $\Z[q,q^{\alpha}]$-module freely generated by vectors $\{ v_0, v_1, \dots\} $, and we endow it with a action of $\mathcal{U}$:
\[ Ev_0 = 0, \ \ Ev_{i+1} = v_i, \ \ Kv_i=q^{\alpha -2i} v_i, \ \ F^{(n)} v_i=\qbinom{n+i}{i}_q \{ \alpha -i ; n \}_q v_{n+i}\]

We define the \textit{topological Verma type module} as the $\widehat{\Z[q,q^{\alpha}]}$-module $\widehat{V^{\alpha}} =  \underset{\underset{n}{\leftarrow}}{\lim} \dfrac{V^{\alpha}}{I_n V^{\alpha}}$ where $I_n V^{\alpha}$ is the ideal generated by elements of the form $\lambda\times v$, $\lambda \in I_n$, $v \in V^{\alpha}$.
Since $J_n V^{\alpha} \subset I_n V^{\alpha}$, we can naturally endow it with a $\mathcal{\tilde{U}}$ module structure.

\medskip
The universal invariant of a knot belongs to the center of the algebra, and acts as a scalar on the Verma type module.
\begin{defn}
If $\mathcal{K}$ is a $0$ framed knot, we define the unified invariant $F_{\infty}(q,q^{\alpha}, \mathcal{K}) \in \widehat{\Z[q,q^{\alpha}]}$ as the coefficient of the scalar action of the universal invariant on the topological Verma type module $\widehat{V^{\alpha}}$:
\[Q^{U_h}(\mathcal{K}) v_0 =Q^{\mathcal{\tilde{U}}} (\mathcal{K}) v_0= F_{\infty}(q,q^{\alpha}, \mathcal{K}) v_0.\]
\end{defn}

\paragraph{The unified invariant:}~\\
Now that we have defined the universal invariant, let us see how it connects to the colored Jones polynomials and the ADO polynomials.

In \cite{Willetts2020AUO}, we proved the following theorem:
\begin{thm}\label{thm_unified_ado_jones}
Let $\mathcal{K}$ be a $0$ framed knot.
\[F_{\infty}(q,q^N,  \mathcal{K})= J_N(q^2, \mathcal{K})\] where $J_N$ is the $N$ colored Jones polynomial associated to the $N+1$ dimensional simple $U_q( \mathfrak{sl}_2)$-module normalised by $J_N(q,unknot)=1$.

\[ F_{\infty} (\zeta_{2r}, q^{\alpha}, \mathcal{K}) =  \frac{ADO_r(q^{\alpha},\mathcal{K})}{A_{\mathcal{K}} (q^{2r\alpha})} \]

where $ADO_r$ is the $r$-th ADO polynomial associated to quantum $\mathfrak{sl}_2$ at root of unity $q=\zeta_{2r}$  normalised by $ADO_r(q^{\alpha},unknot)=1$ and $A_{\mathcal{K}}$ is the Alexander polynomial  with normalisation $A_{unknot}(t)=1$ and $A_{\mathcal{K}}(1)=1$.

\end{thm}

\section{Vassiliev invariants and the ADO polynomials}
\subsection{Vassiliev invariants}
Let $\mathsf{K}$ be the free abelian group generated oriented knots in $S^3$. Let $\mathsf{K}_d$ be the subgroup generated by singular knots with $d$ double points using the identification :

\begin{figure}[h!]
\begin{subfigure}[b]{1\textwidth}
 \centering
  \def\svgwidth{50mm}
    \resizebox{85mm}{!}{
\begingroup%
  \makeatletter%
  \providecommand\color[2][]{%
    \errmessage{(Inkscape) Color is used for the text in Inkscape, but the package 'color.sty' is not loaded}%
    \renewcommand\color[2][]{}%
  }%
  \providecommand\transparent[1]{%
    \errmessage{(Inkscape) Transparency is used (non-zero) for the text in Inkscape, but the package 'transparent.sty' is not loaded}%
    \renewcommand\transparent[1]{}%
  }%
  \providecommand\rotatebox[2]{#2}%
  \newcommand*\fsize{\dimexpr\f@size pt\relax}%
  \newcommand*\lineheight[1]{\fontsize{\fsize}{#1\fsize}\selectfont}%
  \ifx\svgwidth\undefined%
    \setlength{\unitlength}{1069.93055449bp}%
    \ifx\svgscale\undefined%
      \relax%
    \else%
      \setlength{\unitlength}{\unitlength * \real{\svgscale}}%
    \fi%
  \else%
    \setlength{\unitlength}{\svgwidth}%
  \fi%
  \global\let\svgwidth\undefined%
  \global\let\svgscale\undefined%
  \makeatother%
  \begin{picture}(1,0.26181283)%
    \lineheight{1}%
    \setlength\tabcolsep{0pt}%
    \put(0,0){\includegraphics[width=\unitlength,page=1]{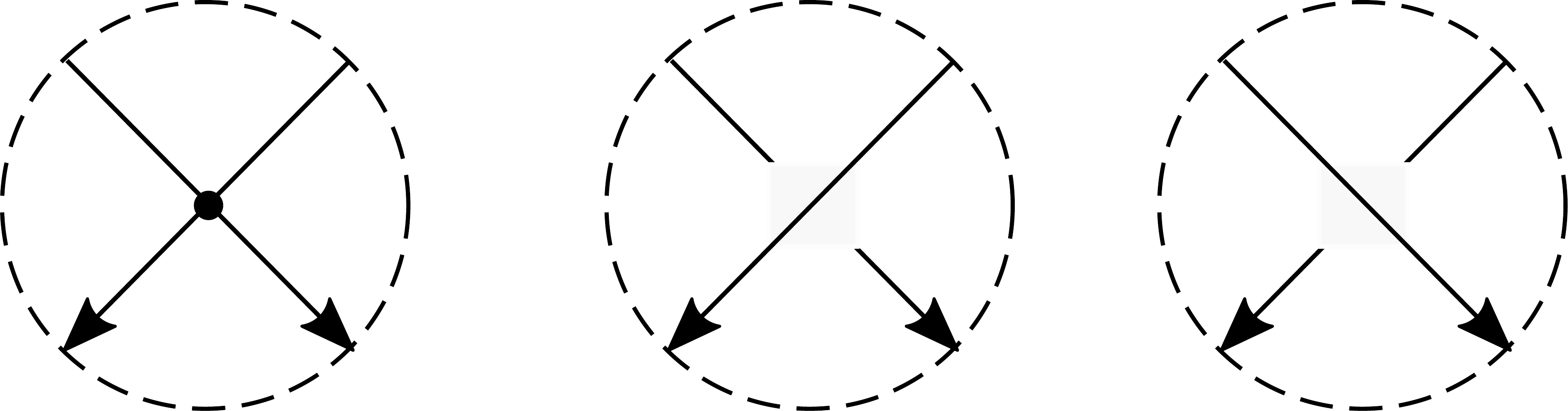}}%
    \put(0.30640231,0.13204497){\color[rgb]{0,0,0}\makebox(0,0)[lt]{\lineheight{1.25}\smash{\begin{tabular}[t]{l}=\\\end{tabular}}}}%
    \put(0.69050112,0.12699501){\color[rgb]{0,0,0}\makebox(0,0)[lt]{\lineheight{1.25}\smash{\begin{tabular}[t]{l}-\\\end{tabular}}}}%
  \end{picture}%
\endgroup%
}
 \end{subfigure}%
   \caption{Relation embedding singular knots in $\mathsf{K}$}
   \label{singular_relation}
 \end{figure}

where, using a diagram of the singular knot,  the pictures in Figure \ref{singular_relation} correspond to a local change at a neighborhood of the double point.

The knots in $\mathsf{K}$ are endowed with the $0$ framing.


A linear map $\nu: \mathsf{K} \to \Q$ is a \textit{Vassiliev invariant of degree $d$} if $\nu|_{\mathsf{K}_{d+1}}=0$.

\begin{note}\hypertarget{NB_degree}
In this definition, a Vassiliev invariant of degree $d$ is also of degree $n$ for all $n \geq d$. This means that, in the rest of the article, a Vassiliev invariant of degree $d$ must be understood as a Vassiliev invariant \textit{of degree $d$ or less}. \\
We do not give the smallest degree for any invariants.
\end{note}

\begin{rem}
Instead of free abelian groups, we could also consider $\Q$ vector spaces.
\end{rem}

\begin{ex}
The colored Jones polynomials, written as $h$-adic power series by setting $q=e^h$, have Vassiliev invariants coefficients. More precisely, if $\mathcal{K}$ is an oriented knot, we endow it with the 0 framing and we have:

\[ J_N(q^2, \mathcal{K})= \sum_d a_{N,d} h^d \] where $a_{N,d}(\mathcal{K}) \in \Q$ and the map $\mathsf{K} \to \Q$, $\mathcal{K} \mapsto a_{N,d}(\mathcal{K}) $ is a Vassiliev invariant of degree\textsuperscript{\hyperlink{NB_degree}{1}} $d$ (see Corollary 7.5 in \cite{ohtsuki2002quantum}).
\end{ex}

Now let's define a broader notion: topological Vassiliev invariants.
\begin{defn}
Let $G$ be an abelian group, $\mathcal{T}$ a topology on $G$.

We say that $v: \mathsf{K} \to G$ is a $\mathcal{T}$-\textit{topological Vassiliev invariant} if $\forall V$ neighborhood of $0$, $\exists N \in \N$ such that $\forall n \ssup N$, $v(\mathsf{K}_n) \subset V$.

\end{defn}

\begin{ex}~
\begin{itemize}
\item A Vassiliev invariant $v: \mathsf{K} \to \Z$ is a discrete topological Vassiliev invariant.
\item The colored Jones invariant $J_N(e^{2h}, \mathcal{K})$ is a $h$-adic topological Vassiliev invariant.
\end{itemize}
\end{ex}

\subsection{Integral power series approach}
We now transpose the unified invariant into a power series ring intermediary between $\widehat{\Z[q,q^{\alpha}]}$ and $ \Q[[h]]$, allowing us on one hand to keep some integral features and evaluate eventually at root of unity, and on the other hand to see it as a series with Vassiliev invariant coefficients.


\begin{prop}
The natural map $\widehat{\Z[q,q^{\alpha}]} \hookrightarrow \Z[[q-1,q^{\alpha}-1]]$ is injective.
\end{prop}

\begin{proof}
See Prop 6.9 of \cite{habiro2007integral}.
\end{proof}

Moreover from Theorem \ref{thm_unified_ado_jones}, we have $F_{\infty}(q, q^{N}, \mathcal{K})= J_N(q^2,\mathcal{K})$, thus we can write the unified invariant in $q^2$: \[F_{\infty}(q, q^{\alpha}, \mathcal{K})= \sum_{n,m \geq 0} b_{n,m}(\mathcal{K}) (q^2-1)^n (q^{2\alpha}-1)^m\] where $b_{n,m} \in \Z$.

\begin{rem}
We can write $q^{-\alpha} =\sum_{m\geq 0} (1-q^{\alpha})^m$.
\end{rem}

\begin{prop}\label{Vassiliev_coeff_unified}
The coefficient $b_{n,m}(\mathcal{K})$ is a Vassiliev invariant of degree\textsuperscript{\hyperlink{NB_degree}{1}} $n+m$.
\end{prop}

\begin{proof}
 Since $F_{\infty}(q, q^{N}, \mathcal{K})= J_N(q^2,\mathcal{K})$, we proceed by induction on $d=n+m$:

\noindent - $b_{0,0}(\mathcal{K})= a_{N,0}(\mathcal{K})$.\\
 - Let $d \in \N$, suppose that for every $n,m$ such that $n+m <d$, $b_{n,m}$ is a Vassiliev invariant of degree\textsuperscript{\hyperlink{NB_degree}{1}} $n+m$.
 
 Now let us understand $a_{N,d}$ in terms of $b_{n,m}$. Remark that if $n+m >d$ then $h^{d+1} | (q-1)^n (q^N -1)^m$, thus $b_{n,m}$ doesn't contribute to $a_{N,d}$ whenever $n+m>d$. Moreover, if we have $n+m=d$, only the first term of the exponential expansion of $(q^2-1)^n (q^{2\alpha}-1)^m $ will contribute to $a_{N,d}$: $ b_{n,m} 2^{n+m} N^m$ for each $n,m$ such that $n+m=d$.
 
 We don't even have to look at the contribution coming from $n+m<d$, indeed on an element $K \in \mathsf{K}_{d+1}$, $b_{n,m}(K)=0$ when $n+m < d$ by induction hypothesis.
 Thus, $\forall N \in \N$ $0=a_{N,d} (K)= \sum_{n+m=d} b_{n,m} 2^{n+m} N^m$ since $a_{N,d}(K)$ is a Vassiliev invariant of degree $d$.\\
 Since this equality holds for all $N \in \N^*$ and since it is a polynomial in $N$, this means that $b_{n,m}(K)=0$ $\forall n,m$ such that $n+m=d$. Hence they are Vassiliev invariants of degree $n+m$.
\end{proof}

\begin{rem}
As a corollary, the unified invariant $F_{\infty}(q,q^{\alpha}, \mathcal{K})$ is a topological Vassiliev invariant for the $((q-1)^n (q^{\alpha}-1)^m)_{n,m \in \N^*}$ filtration topology.

\end{rem}

\subsection{Recovering ADO}

Now we must transpose these results to the $ADO$ polynomials. 

In the $\widehat{\Z[q,q^{\alpha}]}$ setup we could recover $ADO$ by evaluating at roots of unity. We have the result:
\[ F_{\infty}(\zeta_{2r}, q^{\alpha}, \mathcal{K}) = \frac{ADO_r(q^{\alpha}, \mathcal{K})}{A_{\mathcal{K}} (q^{2r\alpha})} \] where $A_{\mathcal{K}} (q^{2r\alpha})$ is the Alexander polynomial of $\mathcal{K}$.

\smallskip

We can also evaluate at roots of unity in the ring $\Z[[q-1, q^{\alpha}-1]]$, but we must be much more careful. Indeed, the codomain of such an evaluation is a $(\zeta_r-1)$-adic completion, which is trivial if $\zeta_r-1$ is invertible, hence we will need some conditions on $r$.

That being said, we will be able to study the factorisation at roots of unity in this setup and prove that the ADO polynomials are $\zeta_r-1$ adic topological Vassiliev invariants. 

We will then study the modulo $r$ case, and show some asymptotic behaviour when $r$ grows.


%

\paragraph{The study of the Alexander polynomial:}~\\
First of all we need to study the the product $F_{\infty}(q,A,\mathcal{K}) \times A_{\mathcal{K}} (q^{2r\alpha})$ as power series with Vassiliev invariant coefficients. We focus this first paragraph on the term $A_{\mathcal{K}} (q^{2r\alpha})$.

\smallskip
Let us write $A_{\mathcal{K}} (q^{2r\alpha})$ as a power series (it exists since it is a Laurent polynomial).

\[ A_{\mathcal{K}} (q^{2r\alpha}) = \sum_{m} \lambda_m (\mathcal{K}) (q^{2r \alpha} -1)^m = \sum_{m} \tilde{\lambda}_m (r,\mathcal{K}) (q^{ 2\alpha} -1)^m.\]

%
%

The fact that $\lambda_m (\mathcal{K})$ are Vassiliev invariants comes from the skein relation verified by the Alexander Polynomial: \[A_{\mathcal{K}_+}(q^{2\alpha})-A_{\mathcal{K}_-}(q^{2\alpha})= \{\alpha\}A_{\mathcal{K}_0}(q^{2\alpha}) \] where $\mathcal{K}_{+,-,0}$ are defined in Figure \ref{alexander_skein}.

\begin{figure}[h!]
\begin{subfigure}[b]{0.33\textwidth}
 \centering
  \def\svgwidth{25mm}
    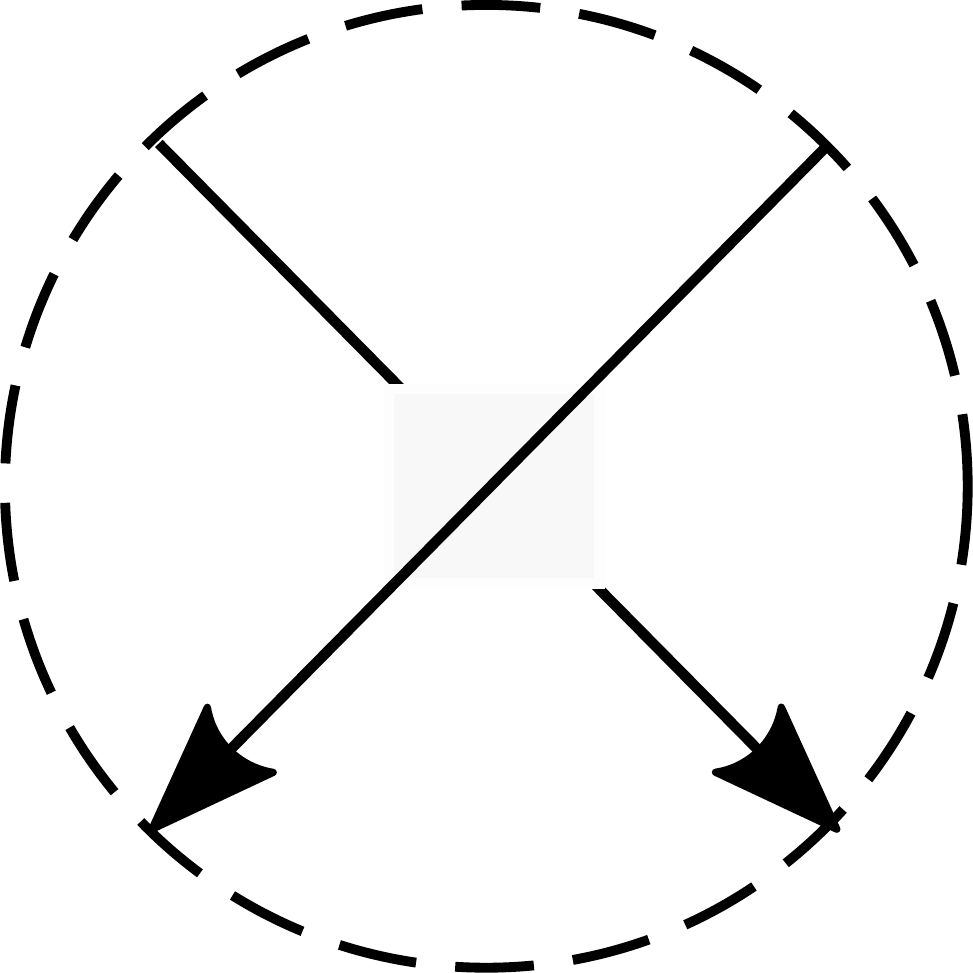
  \label{alexander_skein_pos}
   \caption{$\mathcal{K}_+$}
 \end{subfigure}%
\begin{subfigure}[b]{0.33\textwidth}
 \centering
  \def\svgwidth{25mm}
    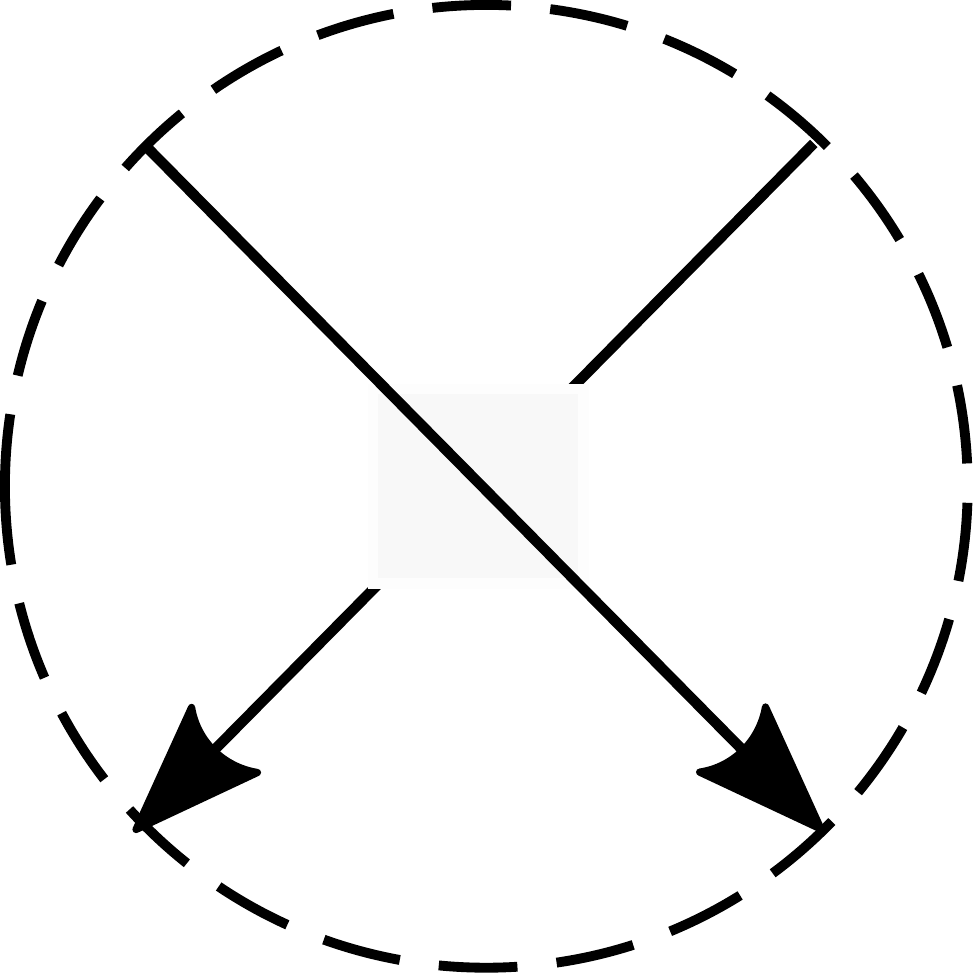
  \label{alexander_skein_neg}
   \caption{$\mathcal{K}_-$}
 \end{subfigure}%
\begin{subfigure}[b]{0.33\textwidth}
 \centering
  \def\svgwidth{25mm}
    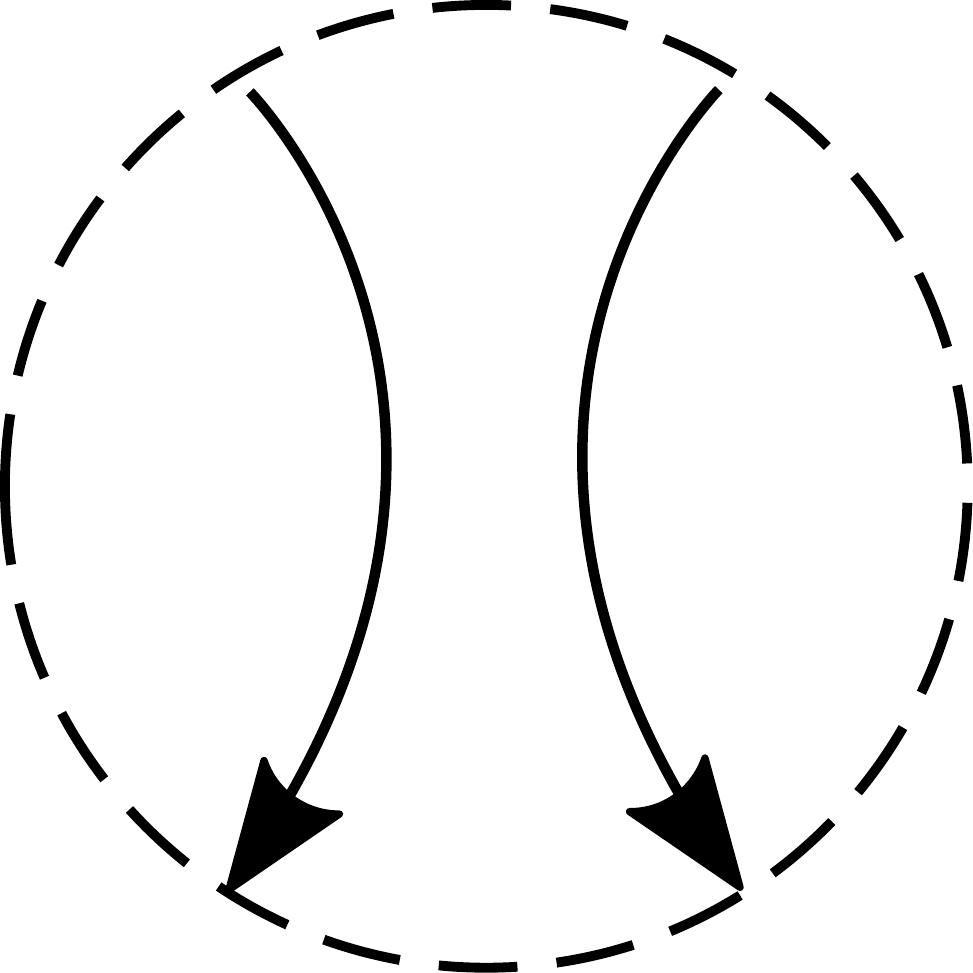
  \label{alexander_skein_0}
   \caption{$\mathcal{K}_0$}
 \end{subfigure}%
 \caption{Local changes for a knot $\mathcal{K}$}
 \label{alexander_skein}
 \end{figure}

\begin{prop}
The coefficient $\lambda_m (\mathcal{K})$ is a Vassiliev invariant of degree\textsuperscript{\hyperlink{NB_degree}{1}} $m$.

\end{prop}

\begin{proof}
From the skein relation we get that for any $K_{d+1}$ singular knot with $d+1$ double points, $\{\alpha\}^{d+1} | A_{K_{d+1}}(q^{2\alpha})$. Hence $\lambda_d (K_{d+1})=0$, $\forall K_{d+1} \in \mathsf{K}_{d+1}$.

\end{proof}

Now we can prove that the $\tilde{\lambda}_m (r,\mathcal{K})$ are also Vassiliev invariants of degree $m$.
\begin{lemma}\label{Vassiliev_lemma_alexander_variable}
The $\tilde{\lambda}_m (r,\mathcal{K})$ are Vassiliev invariants of degree\textsuperscript{\hyperlink{NB_degree}{1}} $m$.
\end{lemma}

\begin{proof}
Since $(q^{2\alpha}-1)^k | (q^{2r\alpha}-1)^k$, $\tilde{\lambda}_m (r,\mathcal{K})$ can be written as a linear combination of $\lambda_k$ for $k \leq m$, and since they all vanish on $\mathsf{K}_{m+1}$, $\tilde{\lambda}_m (\mathcal{K})$ is a Vassiliev invariant of degree $m$.
\end{proof}

\paragraph{The study of the product:}~\\
Now let's study the product $F_{\infty}(q,A,\mathcal{K}) \times A_{\mathcal{K}} (q^{2r\alpha})$ using the multiplicativity of the Vassiliev invariants.

\begin{lemma}\label{lemma_Vassiliev_prod}
If $\mu$ and $\nu$ are Vassiliev invariants of degree $n$ and $m$ respectively, then $\mu \nu$ defined by $\mu \nu (\mathcal{K}) = \mu(\mathcal{K}) \nu(\mathcal{K})$ for any knot $\mathcal{K}$ is a Vassiliev invariant of degree $n+m$.
\end{lemma}

\begin{proof}
See Corollary 3 in \cite{Willerton1998}.
\end{proof}

This means that the coefficients of the product $A_{\mathcal{K}} (q^{2r\alpha}) \times F_{\infty}(q, q^{\alpha}, \mathcal{K})$ are Vassiliev invariant.

 More precisely we have:

\begin{align*}
A_{\mathcal{K}} (q^{2r\alpha}) \times F_{\infty}(q, q^{\alpha}, \mathcal{K}) &= \sum_{m,n} (\sum_{k=0}^m \tilde{\lambda}_k (r,\mathcal{K}) b_{n,m-k}(\mathcal{K})) (q^2-1)^n (q^{2\alpha}-1)^m\\
&=\sum_{m,n} c_{n,m}(r,\mathcal{K}) (q^2-1)^n (q^{2\alpha}-1)^m
\end{align*}

where $c_{n,m}= \sum_{k=0}^m \tilde{\lambda}_k b_{n,m-k}$ are Vassiliev invariant of degree $n+m$.

\paragraph{The ADO polynomial:}~

Let's now study the factorisation at roots of unity. First we need to define the evaluation map.
Let \[ \Z[[\zeta_r-1]]:= \underset{\underset{n}{\leftarrow}}{\lim} \ \Z[\zeta_r]/(\zeta_r-1)^n.\]
The natural map: \[ j : \Z[\zeta_r] \to \Z[[\zeta_r-1]] \] is injective iff $\underset{n \in \N^*}{\cap}(\zeta_r-1)^n=\{0\}.$

Moreover, we have a well defined surjective map \[ev_r : \Z[[q-1, q^{\alpha}-1]] \to \Z[[\zeta_r-1, q^{\alpha}-1]].\]

Then, one gets the equality : 
\begin{align*}
 j(ADO_r(q^{\alpha}, \mathcal{K}) ) &= ev_r( A_{\mathcal{K}} (q^{2r\alpha})  F_{\infty}(q, q^{\alpha}, \mathcal{K}))
\end{align*}

It is essential to study for which $r$ the map $j$ is injective.

%
%

%
%

\begin{lemma}\label{cyclotomic_lemma_inverse}
$(\zeta_{r}-1)$ is not invertible in $\Z[\zeta_{r}]$ if and only if $r$ is  a power of a prime number.
More precisely, if $r=p^l$ is a power of a prime number $((\zeta_{r}-1)^{\varphi(r)})=(r)$ as ideals in $\Z[\zeta_{r}]$, where $\varphi(r)$ is the Euler phi function.
\end{lemma}

\begin{proof}
See Lemma 1.4 and Prop 2.8 in \cite{washington1997introduction}.
\end{proof}

\begin{cor}~
\begin{itemize}
\item $\Z[[\zeta_r-1, q^{\alpha}-1]]$ is non trivial iff $r$ is a power of a prime number.
\item $j$ is injective iff $r$ is a power of a prime number.
\end{itemize}
\end{cor}
\bigskip

Let us suppose now that $r=p^l$ is a power of a prime number, we have $j:\Z[\zeta_r] \subset \Z[[\zeta_r-1, q^{\alpha}-1]] $, we will omit the $j$.

Hence we have the following equality in $\Z[[\zeta_r-1, q^{\alpha}-1]]$: 
\[ADO_r(q^{\alpha}, \mathcal{K})= \sum_{m \geq 0} \sum_{n \geq 0} c_{n,m}(r,\mathcal{K}) (\zeta_{r} -1 )^n (q^{2\alpha}-1)^m\] and $c_{n,m}(r,\mathcal{K})$ is a Vassiliev invariant of degree $n+m$.

Thus we have the following theorem.

\begin{thm}\label{thm_ado_Vassiliev}
Let $\mathcal{K}$ be a $0$ framed knot in $S^3$ and $r=p^l$ be a power of a prime number.

There exist Vassiliev invariants $c_{n,m}(r,\mathcal{K})$ of degree\textsuperscript{\hyperlink{NB_degree}{1}} $n+m$, such that we can write \[ADO_r(q^{\alpha}, \mathcal{K}) = \sum_{m \geq 0} \sum_{n \geq 0} c_{n,m}(r,\mathcal{K}) (\zeta_r-1)^n (q^{2\alpha}-1)^m  \] in $\Z[[\zeta_r-1, q^{\alpha}-1]]$.

This means that the $ADO$ polynomials $ADO_r(q^{\alpha}, \mathcal{K}) \in \Z[\zeta_r, q^{\alpha}]$ are topological Vassiliev invariants for the filtration $((\zeta_r-1)^n (q^{\alpha}-1)^m)_{n,m\in\N*}$.

\end{thm}

Nevertheless, this form is not unique and happens in the completion ring \(\Z[[\zeta_r-1, q^{\alpha}-1]].\)

\paragraph{The $r$-adic form:}
We can compute a unique expansion for the ADO polynomials with $r$-adic topological invariants.

Let $\varphi$ be the Euler phi function.\\ For a power of a prime $r=p^l$, we have \[\varphi(r)=(p-1)p^{l-1}. \]

The ADO polynomial can be uniquely written as:
\[ADO_r(q^{\alpha}, \mathcal{K})= \sum_{m \geq 0} \sum_{n=0}^{\varphi(r)-1} d_{n,m}(r,\mathcal{K}) (\zeta_{r} -1 )^n (q^{2\alpha}-1)^m\]
where $d_{n,m} \in \Z$.

This comes from the fact that since $1, \zeta_{r}, \dots, \zeta_{r}^{\varphi(r)-1}$ form a basis of $\Z[\zeta_{r}]$ as a $\Z$ module, so is $1,  \zeta_{r}-1, \dots, (\zeta_{r}-1)^{\varphi(r)-1}$.

Let's now define $\Z_r := \underset{\underset{n}{\leftarrow}}{\lim} \ \Z/(r)^n$ the ring of $r$ adic integers.

Since $((\zeta_r-1)^{n(\varphi(r)-1)})=(r^n)$, we have an injective map \[i: \Z[[\zeta_r-1, q^{\alpha}-1]] \to \Z_r[\zeta_r] [[q^{\alpha}-1]].\]

Now, if we write 
\begin{align*}
ADO_r(q^{\alpha}, \mathcal{K}) &=\sum_{m,n \geq 0} c_{n,m}(r,\mathcal{K}) (\zeta_r-1)^n (q^{2\alpha}-1)^m\\
&= \sum_{m \geq 0} \sum_{i=0}^{\varphi(r)-1} \sum_{j \geq 0} CL_{j,i,m}(r,\mathcal{K}) r^j (\zeta_r-1)^i (q^{2\alpha}-1)^m \end{align*} 
where $CL_{j,i,m}(r,\mathcal{K})$ is a linear combination of $(C_{n,m}(r,\mathcal{K}))_n$.

Since, $\forall n \geq (j+1)\varphi(r)$, $C_{n,m}(r,\mathcal{K})(\zeta_r-1)^n \in r^{j+1} \Z[\zeta_r]$, then $CL_{j,i,m}(r,\mathcal{K})$ is a linear combination of $(C_{n,m}(r,\mathcal{K}))_{n \sinf (j+1)\varphi(r)}$. 

Hence, $CL_{j,i,m}(r,\mathcal{K})$ is a Vassiliev invariant of degree $(j+1)\varphi(r)+m$. 

\begin{thm}\label{thm_ado_Vassiliev_radic}
Let $\mathcal{K}$ be a $0$ framed knot in $S^3$ and $r=p^l$ a power of a prime number.

We can write in a unique way \[ADO_r(q^{\alpha}, \mathcal{K}) = \sum_{m \geq 0} \sum_{n=0}^{\varphi(r)-1} d_{n,m}(r,\mathcal{K}) (\zeta_r-1)^n (q^{2\alpha}-1)^m  \] in $\Z[\zeta_r, q^{\alpha}]$,
and we have  
\begin{enumerate}
\item $d_{n,m}(r,\mathcal{K}) \mod r^j$ is a Vassiliev invariant of degree\textsuperscript{\hyperlink{NB_degree}{1}} $(j+1)\varphi(r)+m$.
\item In consequence, $d_{n,m}(r,\mathcal{K}) \in \Z$ and $d_{n,m}(r,\mathcal{K})$ is a $r$-adic topological Vassiliev invariant.
\item Thus, $ADO_r(q^{\alpha}, \mathcal{K})\in \Z[\zeta_r, q^{\alpha}]$ is a topological Vassiliev invariant for the filtration $(r^j (q^{\alpha}-1)^m)_{j,m \in \N}$ topology.
\end{enumerate}

\end{thm}
\medskip

\begin{proof}
\begin{enumerate}
\item $d_{n,m}(r,\mathcal{K})= \sum_{l=0}^{j-1} CL_{l,n,m}(r,\mathcal{K}) r^l \mod r^j$ by the uniqueness of the decomposition, hence is a Vassiliev invariant of degree $(j+1)\varphi(r)+m$.
\item  Let $m\in\N$ and $j \in \N$, if $K \in \mathsf{K}_{d}$ for $d \ssup (j+1)\varphi(r)+m$, $d_{n,m}(r,K) \in r^j \Z$, hence $d_{n,m}(r,K)$ is an $r$-adic topological Vassiliev invariant.
\item Let $m,j \in \N$, since  $\forall K \in \mathsf{K}_{d}$ for $d \ssup (j+1)\varphi(r)+m$, $\forall l\leq m$ $d_{n,l}(r,K) \in r^j \Z$, one gets \[ADO_r(q^{\alpha},K) \in r^j \Z[\zeta_r, q^{\alpha}]+ (q^{\alpha}-1)^m \Z[\zeta_r, q^{\alpha}].\]
\end{enumerate}
\end{proof}

Here we got a unique expansion, but the price to pay is that the coefficients are, a priori, no longer Vassiliev invariant but $r$ adic topological Vassiliev invariants.

\subsection{Study of $\overline{d_{n,m}(r,\mathcal{K})}$ and ADO asymptotic behavior mod r}
Lets focus on the previous setup but modulo $r$. We will study the dependence of $c_{n,m}$ and $d_{n,m}$ in $r$, and deduce some asymptotic behavior of ADO modulo $r$.

In the ring $(\Z[\zeta_{2r}]/(r))[[q^{\alpha}-1]]$ we have the following equality:
\[\sum_{m \geq 0} \sum_{n=0}^{\varphi(r)-1} \overline{c_{n,m}(\mathcal{K})} (\zeta_{r}-1)^n (q^{2\alpha}-1)^m = \sum_{m \geq 0} \sum_{n=0}^{\varphi(r)-1} \overline{d_{n,m}(r,\mathcal{K})} (\zeta_{r} -1 )^n (q^{2\alpha}-1)^m \]

where $\overline{c_{n,m}}, \overline{d_{n,m}} \in \Z / r\Z$.

\begin{prop}\label{thm_ADO_Vassiliev_r}
Let $r=p^l$ be a power of a prime number.\\
We can write in a unique way \[ADO_r(q^{\alpha}, \mathcal{K})= \sum_{m \geq 0} \sum_{n=0}^{\varphi(r)-1} d_{n,m}(r,\mathcal{K}) (\zeta_{r} -1 )^n (q^{2\alpha}-1)^m\]
And, for any $m \in \N$ and any $1 \leq n \leq \varphi(r)-1$, the map
\begin{align*}
 \nu_{n,m} : \mathsf{K} &\to \Z/r\Z\\
 \mathcal{K} &\mapsto \overline{d_{n,m}(r,\mathcal{K})}
 \end{align*}
 
 is a Vassiliev invariant of degree\textsuperscript{\hyperlink{NB_degree}{1}} $n+m$.

\end{prop}

We may now study the dependence of $d_{n,m}(r,\mathcal{K})$ in the variable $r$.

Recall that : \[F_{\infty}(q, q^{\alpha}, \mathcal{K})= \sum_{n,m \geq 0} b_{n,m}(\mathcal{K}) (q^2-1)^n (q^{2\alpha}-1)^m,\]
\[ A_{\mathcal{K}} (q^{2r\alpha}) = \sum_{m} \lambda_m (\mathcal{K}) (q^{2r \alpha} -1)^m = \sum_{m} \tilde{\lambda}_m (\mathcal{K}) (q^{ 2\alpha} -1)^m.\]

\begin{rem} $b_{n,m}(\mathcal{K})$ and $\lambda_m (\mathcal{K})$ don't depend on $r$, but $\tilde{\lambda}_m (r,\mathcal{K})$ does.
\end{rem}

Let's compute $\tilde{\lambda}_m (r,\mathcal{K})$:

\begin{align*}
A_{\mathcal{K}} (q^{2r\alpha})  &= \sum_{m \geq 0} \lambda_m (\mathcal{K}) (q^{2r \alpha} -1)^m\\
								&= \sum_{m \geq 0} \lambda_m \sum_{k \geq 0} \binom{m}{k} (-1)^{m-k} q^{2rk \alpha}\\
								&= \sum_{m \geq 0} \lambda_m \sum_{k \geq 0} \binom{m}{k} (-1)^{m-k} \sum_{j \geq 0} \binom{rk}{j} (q^{2\alpha}-1)^j \\
								&= \sum_{j \geq 0} (\sum_{m \geq 0} \lambda_m \sum_{k \geq 0} \binom{m}{k} (-1)^{m-k}  \binom{rk}{j}) (q^{2\alpha}-1)^j
\end{align*} 

\begin{lemma}
$\forall m \ssup j$, $\sum_{k \geq 0} \binom{m}{k} (-1)^{m-k}  \binom{rk}{j}=0$.
\end{lemma}
\begin{proof}
Let $F_j(X,m)= \sum_{k \geq 0} \binom{m}{k} (-1)^{m-k} \binom{rk}{j} X^{rk-j}$, we study the $X$ derivatives: 
\begin{align*}
\frac{F_0^{(n)}(X,m)}{(n)!}&=\sum_{k \geq 0} \binom{m}{k} (-1)^{m-k} \frac{(rk)!}{(rk-n)! (n)!} X^{rk-n}\\
		&= F_n(X,m).
\end{align*} 

Since \[F_0(X,m)= \sum_{k \geq 0} \binom{m}{k} (-1)^{m-k} X^{rk} = (X^r -1)^m,\] then \[F_0^{(j)}(1,m) =0 \ \forall m \ssup j \]

Hence, $\forall m \ssup j$, $F_j(1,m)=0$.
\end{proof}

We can thus write $\tilde{\lambda}_j (r,\mathcal{K})= \sum_{m \geq 0} \lambda_m \sum_{k \geq 0} \binom{m}{k} (-1)^{m-k}  \binom{rk}{j}$.

In particular, $\tilde{\lambda}_j (r,\mathcal{K}) \in r \Z$, $ \forall 0 \sinf j \sinf r$.


\begin{prop}\label{thm_asymptotic_ado}
Let $r=p^l$ be a power of a prime number, $\mathcal{K}$ a $0$ framed knot in $S^3$.
\begin{enumerate}
\item $\forall n \in \N$, $ \forall m \sinf r$, $\overline{d_{n,m}(r,\mathcal{K})}= \overline{b_{n,m}(\mathcal{K})}$ in $\Z/rZ$,
\item $\Val_{q^{\alpha}-1}(\overline{F_{\infty}(\zeta_r,q^{\alpha},\mathcal{K})- ADO_r(q^{\alpha}, \mathcal{K})}) \underset{r \to +\infty}{\longrightarrow} +\infty$ where \[\Val_{q^{\alpha}-1}(\sum_m a_m (q^{\alpha}-1)^m)=\min(m \in \N | \  a_m \neq 0).\]
\end{enumerate}

\end{prop}

\begin{proof}
The first point comes from the fact that \[A_{\mathcal{K}} (q^{2r\alpha}) \times F_{\infty}(q, q^{\alpha}, \mathcal{K}) = \sum_{m,n \geq 0} (\sum_{k=0}^m \tilde{\lambda}_k (r,\mathcal{K}) b_{n,m-k}(\mathcal{K})) (q^2-1)^n (q^{2\alpha}-1)^m,\] hence for all $m \sinf r$  \[\overline{d_{n,m}(r,\mathcal{K})}= \overline{\sum_{k=0}^m \tilde{\lambda}_k (r,\mathcal{K}) b_{n,m-k}(\mathcal{K})}=\overline{ \tilde{\lambda}_0 (r,\mathcal{K})b_{n,m}(\mathcal{K})}\] since $\tilde{\lambda}_m (r,\mathcal{K}) \in r \Z$, $ \forall 0 \sinf j \sinf r$. Since $A_{\mathcal{K}}(1)=1$, then $\tilde{\lambda}_0(r,\mathcal{K})=1$.

The second is a direct application of the first, since $m \sinf r$,  $\overline{d_{n,m}(r,\mathcal{K})}=\overline{b_{n,m}(\mathcal{K})}$, \[\overline{F_{\infty}(\zeta_r,q^{\alpha},\mathcal{K})- ADO_r(q^{\alpha}, \mathcal{K})}) \in (q^{2\alpha}-1)^r (\Z[\zeta_{2r}]/(r)) [[q^{\alpha}-1]]. \]
\end{proof}

As a corollary of Proposition \ref{thm_asymptotic_ado}, we have that $\overline{d_{n,m}(r, \mathcal{K})}$ are not all trivial Vassiliev invariants.

\begin{cor}
If $b_{n,m}(\mathcal{K}) \neq 0$ then for any big enough $r=p^l$ power of a prime,  \[\overline{d_{n,m}(r,\mathcal{K})} \neq 0 \text{ in } \Z/r\Z. \]
\end{cor}

\bibliographystyle{abbrv}
\bibliography{finite_type_ADO.bib}
\address
\end{document}